\documentclass[12pt]{amsart}

\calclayout

\usepackage[all,cmtip]{xy}
\let\objectstyle=\displaystyle

\usepackage[utf8]{inputenc}
\usepackage{amsmath}
\usepackage{amssymb}
\usepackage{amsthm}
\usepackage{amsrefs}
\usepackage{float}
\usepackage{longtable}
\usepackage{tikz}
\usepackage{bookmark}
\usepackage{hyperref}
\usepackage[capitalise]{cleveref}
\usepackage[foot]{amsaddr}

\newtheorem{theorem}{Theorem}[section]
\newtheorem{lemma}[theorem]{Lemma}
\newtheorem{proposition}[theorem]{Proposition}
\newtheorem{corollary}[theorem]{Corollary}
\newtheorem{remark}[theorem]{Remark}
\newtheorem{definition}[theorem]{Definition}

\newcommand{\F}{\mathbb{F}}
\newcommand{\N}{\mathbb{N}}
\newcommand{\Z}{\mathbb{Z}}
\newcommand{\mc}[1]{\mathcal{#1}}
\newcommand{\Sq}{\tilde{\mathcal{S}}_q}
\newcommand{\Seight}{\tilde{\mathcal{S}}_8}
\newcommand{\X}{\mathcal{X}}
\newcommand{\Mod}[1]{\ (\mathrm{mod}\ #1)}

\title{Two-point AG codes from one of the Skabelund maximal curves}

\author[Leonardo Landi]{Leonardo Landi$^1$} \address{$^1$Independent researcher - ORCID 0000-0003-4835-5321} 
\author[Marco Timpanella]{Marco Timpanella$^2$} \address{$^2$Dipartimento di Matematica e Informatica, Universit\`{a} degli Studi di Perugia, Perugia, 06100, Italy - ORCID 0000-0001-6058-1909} \email{marco.timpanella@unipg.it} \thanks{}
\author[Lara Vicino]{Lara Vicino$^3$} \address{$^3$Department of Applied Mathematics and Computer Science, Technical University of Denmark, Kongens Lyngby 2800, Denmark - ORCID 0000-0002-7480-8681}  \email{lavi@dtu.dk} \thanks{}

\date{}

\begin{document}

\begin{abstract}
In this paper, we investigate two-point Algebraic Geometry codes associated to the Skabelund maximal curve constructed as a cyclic cover of the Suzuki curve. In order to estimate the minimum distance of such codes, we make use of the generalized order bound introduced by P. Beelen and determine certain two-point Weierstrass semigroups of the curve.
\end{abstract}

\maketitle

\vspace{0.5cm}\noindent {\em Keywords}: AG codes, Skabelund curve, order bound, two-point code, two-point Weierstrass semigroup

\vspace{0.2cm}\noindent{\em MSC}: Primary: 11G20. Secondary: 11T71, 94B27, 14H50, 14G50

\vspace{0.2cm}\noindent

\section{Introduction}

Let $\mathbb{F}_{q}$ be a finite field with $q$ elements, where $q$ is a power of a prime $p$. A projective, absolutely irreducible, non-singular, algebraic curve $\X$ defined over $\F_q$ is called $\F_q$\emph{-maximal} if the number $|\X(\F_q)|$ of its $\F_q$-rational points attains the Hasse-Weil upper bound
$$
|\X(\F_q)| \le q + 1 + 2g\sqrt{q},
$$
where $g$ is the genus of the curve.  Notable examples of maximal curves include the Hermitian, Suzuki, and Ree curves, collectively known as Deligne-Lusztig curves. The Kleiman-Serre covering result plays a key role in constructing new examples of maximal curves. It states that
 any curve defined over $\F_q$ and $\F_q$-covered by an $\F_q$-maximal curve is itself $\F_q$-maximal. Actually, apart from the Deligne-Lusztig curves, most of the known $\F_{q^2}$-maximal curves are obtained, following this approach, from the Hermitian curve. Initially, there was speculation that all maximal curves could be covered by the Hermitian curve. However, this was proven false in 2009 by M. Giulietti and G. Korchm\'{a}ros \cite{GK}, who constructed an $\F_{q^6}$-maximal curve not covered by the Hermitian curve when $q>2$. This curve, known as the GK curve, was constructed as a Galois cover of the Hermitian curve. Since then, a few other examples of maximal curves not covered by the Hermitian curve have been provided \cite{BM,ggs10,TTT}. Reproducing the way in which the GK curve was constructed from the Hermitian curve, in \cite{S17} D. Skabelund constructed new infinite families of maximal curves as cyclic covers of the two other Deligne-Lusztig curves. In this paper, we deal with the Skabelund $\F_{q^4}$-maximal curve $\Sq$ arising from the Suzuki curve.

Maximal curves are interesting objects in their own right, but also for their applications in Coding Theory. The essential idea going back to Goppa's work \cite{Goppa1} is that a linear code (the so called Algebraic Geometry code) can be obtained from an algebraic curve $\X$ defined over $\mathbb{F}_q$ by evaluating certain rational functions, whose poles are prescribed by a given $\mathbb{F}_q$-rational divisor $G$, at some $\mathbb{F}_q$-rational divisor $D$ whose support is disjoint from that of $G$. Algebraic Geometry codes (AG codes for short) are proven to have good performances, provided that $\mathcal{X}$, $G$ and $D$ are carefully chosen in an appropriate way. In particular, as the relative Singleton defect of an AG code from a curve $\X$ is upper bounded by the ratio $g/N$, where $g$ is the genus of $\X$ and $N$ can be as large as the number of $\mathbb F_q$-rational points of $\X$, it follows that curves with many rational points with respect to their genus are of great interest in Coding Theory. For this reason, AG codes from maximal curves have been widely investigated in the last years, see for instance \cite{BMZ2, BMZ1,KNT,TC}. Regarding the choice of the two divisors $D$ and $G$, the latter is typically taken to have support at one or two points of $\X$. Such codes are known as one-point codes and two-point codes, respectively, and have been extensively investigated, see for instance \cite{ CT,LT,GF,LV22,TwoPointsSuzuki,MTZ}. 

In this paper, we investigate duals of two-point codes from the curve $\Sq$. In order to estimate their minimum distance, we make use of the generalized order bound introduced by P. Beelen in \cite{OrderBound}. This bound constitutes a generalization to multi-point AG codes of the order bound for one-point AG codes (see \cite{AGcodes}) and has been further improved in \cite{DK09} and \cite{DP10}.

The paper is organized as follows. Section \ref{Prel} provides a review of the relevant results on AG codes, the generalized order bound, and the curve $\Sq$. Section \ref{sec:two_point_semigroup} is devoted to the computation of a two-point Weierstrass semigroup of $\Sq$. Finally, in Section \ref{sec:results}, we analyze the performance of the constructed codes. Notably, in several cases, the order bound offers a more accurate estimate of the minimum distance compared to the Goppa bound. Moreover, the minimum distance of all the two-point codes constructed is always at least that of the one-point code of the same dimension. 

\section{Preliminaries}\label{Prel}

In this section, we start by introducing the notations and terminology that we use for AG codes, see \cite{Stichtenoth}*{Chapter 2} for further details. Moreover, we recall the definition of the generalized order bound and the related results that we use in order to estimate the minimum distance of the two-point codes studied in the paper. A thorough explanation of these results can be found in \cite{TwoPointCodes} and \cite{OrderBound}. In particular, a detailed description of the algorithm used for estimating the distance of the codes can be found in \cite{LV22}*{Section 4}.
Finally, we recall the definition and some results regarding the Skabelund curve $\Sq$. 

Throughout this section, let $q = p^h$, for $p$ a prime number and $h>0$ an integer. Let $\mathcal{X}$ be a projective absolutely irreducible non-singular algebraic curve defined over the finite field $\mathbb{F}_q$ of genus $g(\mathcal{X})$. We denote as $\mathbb{F}_q(\mathcal{X})$ the function field of $\X$ and as $\X(\F_q)$ the set of $\F_q$-rational points of $\X$. 

\subsection{Algebraic Geometry codes}
 
Let $P_1, \dots, P_n\in \mathcal{X}(\F_q)$ be distinct points and let $D$ be the divisor $D := P_1 + \dots + P_n$. Moreover, choose $G$ to be another divisor on $\X$ with support disjoint from the support of $D$. Then,
the AG code $C_L(D, G)$ is defined as 
$$ C_L(D, G) := \{ (f(P_1), \dots, f(P_n)) \mid f \in L(G) \},$$
where 
$$ L(G) := \{ f \in \mathbb{F}_q(\mathcal{X})\setminus \{0\} \mid (f) + G \geq 0 \} \cup \{ 0 \}$$
is the Riemann-Roch space associated to the divisor $G$.
The code $C_L(D, G)$ is a linear subspace of $\mathbb{F}_q^n$ of dimension $\mathrm{dim}(C_L(D, G)) = \mathrm{dim}(L(G)) - \mathrm{dim}(L(G-D))$ and its minimum distance $d$ satisfies the Goppa bound $d \geq n - \mathrm{deg}(G)$. 

The dual code $C_L(D, G)^\perp$ of $C_L(D, G)$ is the orthogonal complement of $C_L(D, G)$ in $\F_{q}^n$, hence it is a linear subspace of $\mathbb{F}_q^n$ of dimension $n - \mathrm{dim}(C_L(D, G))$. It can be shown that the dual of an AG code is itself an AG code (see \cite{Stichtenoth}*{Proposition 2.2.10}), hence the Goppa bound gives also a lower bound for the minimum distance $d^\prime$ of $C_L(D, G)^\perp$, namely $d^\prime \geq \mathrm{deg}(G) - 2g(\mathcal{X}) + 2$.

An AG code $C_L(D, G)$ is said to be a two-point code if the support of the divisor $G$ consists of precisely two points $Q_1$ and $Q_2$, namely $G = aQ_1 + bQ_2$ for some $a, b \in \mathbb{Z}\setminus \{0\}$. It does not necessarily hold that the dual of a two-point code $C_L(D, G)$ is itself a two-point code. However, since the literature on the subject typically refers also to duals of two-point codes just as \emph{two-point codes}, throughout the paper we will generally adopt the same slightly abusive terminology.

\subsection{The generalized order bound}
\label{subsec:orderbound}
Since the Goppa bound for the minimum distance of a code $C_L(D, G)^\perp$ is often not sharp,  more refined bounds for the dual minimum distance of AG codes have been introduced in the literature. In this paper, we will make use of the generalized order bound introduced by P. Beelen in \cite{OrderBound}, which constitutes a generalization of the Feng-Rao bound for one-point AG codes (see \cite{AGcodes}). Further improvements to the generalized order bound can be found in \cite{DK09} and \cite{DP10}. 
For a detailed survey on the relations between other bounds and the generalized order bound introduced in \cite{OrderBound}, we refer to \cite{DKP11}. 

We now recall the definition of the generalized order bound and discuss its use to practically estimate the minimum distance of two-point codes. We start with the following definition. 
\begin{definition}
\label{def:Gnongaps}
Let $G$ be a divisor on $\X$. For $R\in \X(\F_q)$, the set of \textit{$G$-non-gaps at $R$} (see \cite{ConsecutiveGaps}) is defined as  
\begin{align*}
  H(R; G) & := \left\{ -v_R(f) \ \Biggl\lvert \ f \in \bigcup_{i=-\mathrm{deg}(G)}^{\infty}L(G +  iR) \setminus \{ 0 \}\right\}.
\end{align*}
Moreover, define
\begin{align*}
  N(R; G) & := \{ (i, j) \in H(R; 0) \times H(R; G) \mid i+j = v_R(G)+1 \}, \\
  \nu(R; G) & := \# N(R; G),
\end{align*}
where $v_R(G)$ denotes the coefficient of $R$ in the divisor $G$.
\end{definition}

\begin{remark}
    From Definition \ref{def:Gnongaps}, it is immediate to see that
    \begin{itemize}
        \item $H(R; G + R) = H(R; G)$
        \item $H(R; 0)$ is the Weierstrass semigroup $H(R)$ at $R$.
    \end{itemize}  
\end{remark}

\begin{definition}[\cite{OrderBound}*{Definition 6}]
  \label{def:orderbound}
  Let $D$ be a divisor that is a sum of $n$ distinct rational points of $\mathcal{X}$ and $G$ a divisor on $\mathcal{X}$ such that all the points in its support are rational. Further, suppose that the support of $G$ is disjoint from the support of $D$. For any infinite sequence $S = R_1, R_2,\ldots$ of points of $\mathcal{X}\setminus \mathrm{supp}(D)$ define
  $$ d_S(G) := \mathrm{min}\{\nu(R_{i+1}; G + R_1 + \cdots + R_i)\}, $$
  where the minimum is taken over all $i\geq 0$ such that $L(G + R_1 + \cdots + R_i) \neq L(G + R_1 + \cdots + R_{i+1})$. Moreover, define 
  $$ d(G):=\mathrm{max} \ d_S(G), $$
  where the maximum is taken over all infinite sequences $S$ of points having entries in $\mathcal{X}\setminus \mathrm{supp}(D)$.
\end{definition}

The following proposition, a consequence of \cite{OrderBound}*{Theorem 7}, ensures that the integer $d(G)$ is in fact a lower bound for the minimum distance $d^\prime$ of the code $C_L(D, G)^\perp$.

\begin{proposition}
  \label{prop:orderbound}
  Let $C_L(D, G)$ be an AG code with $D$ and $G$ as in Definition \ref{def:orderbound}. Then the minimum distance $d^\prime$ of the dual code $C_L(D, G)^\perp$ satisfies the inequality $d^\prime \geq d(G)$.
\end{proposition}

We hence refer to $d(G)$ as the \emph{generalized order bound} (or, for short, simply as the \textit{order bound}) for the minimum distance of $C_L(D, G)^\perp$.

\begin{remark}
\label{rem:restrictions}
    It is not difficult to see (\cite{LV22}*{Lemma 2.3}) that $d(G)\geq \mathrm{deg}(G)-2g(\X)+2$ and that equality holds if $\mathrm{deg}(G) \geq 4g(\X) - 1$. This observation is useful for practical purposes. 
In fact, only the ?rst $4g-1-\deg(G)$ entries from a sequence $S$ are relevant to de?ne $d(G)$, although, for practical reasons, one needs to select a finite number of entries, at the cost of obtaining a possibly worse bound.

In order to compute the order bound $d(G)$ for the minimum distance of the two-point code $C_L(D, aQ_1 + bQ_2)^\perp$ with $Q_1,Q_2\in \X(\F_q)$ distinct points, $a,b\in \mathbb{Z}\setminus \{0\}$ and $D$ the sum of all the points in $\X(\F_q)\setminus\{Q_1,Q_2\}$, we will consider sequences $S = R_1, R_2, \dots$ of length at most $4g(\X)-1-\mathrm{deg}(G)$ and $R_i \in \{ Q_1, Q_2 \}$ for all $i \geq 1$.
\end{remark}

With slight abuse of notation, although applying the restrictions discussed in Remark \ref{rem:restrictions},
we will continue to denote the bound with $d(G)$ and to refer to it as the \textit{order bound}. Note that this choice does not affect the statement of Proposition \ref{prop:orderbound}.

In the case of a two-point code $C_L(D, aQ_1 + bQ_2)$, with notations and constraints as in Remark \ref{rem:restrictions}, the fundamental ingredient for the computation of the order bound for the minimum distance of $C_L(D, aQ_1 + bQ_2)^\perp$ is then an efficient way for computing the integers 
\begin{alignat*}{2}
  \nu(Q_1; aQ_1 + bQ_2) & = \# \{ (i, j) \in H(Q_1) \times H(Q_1; bQ_2) \mid i+j = a+1 \}, \\
  \nu(Q_2; aQ_1 + bQ_2) & = \# \{ (i, j) \in H(Q_2) \times H(Q_2; aQ_1) \mid i+j = b+1 \}.
\end{alignat*}
Therefore, we recall here a useful method for computing the sets $H(Q_1; bQ_2)$ and $H(Q_2; aQ_1)$, which has also been used in \cite{TwoPointCodes} and \cite{LV22}. We start by giving the definition of two-point Weierstrass semigroup, as studied in \cite{BT06}.

\begin{definition}
\label{def:ringregular}
Denote with $\mathcal{R}(Q_1, Q_2)$ the ring of functions in $\mathbb{F}_q(\mathcal{X})$ that are regular outside $Q_1$ and $Q_2$, namely
$$ \mathcal{R}(Q_1, Q_2) := \{ f \in \mathbb{F}_q(\mathcal{X}) \mid v_R(f) \geq 0 \; \forall R \neq Q_1, Q_2 \}. $$
The two-point Weierstrass semigroup $H(Q_1, Q_2)$ is defined as the set
$$ H(Q_1, Q_2) := \{ (i, j) \in \mathbb{Z}^2 \mid \exists f \in \mathcal{R}(Q_1, Q_2) \setminus \{0\}, v_{Q_1}(f) = -i, v_{Q_2}(f) = -j \}. $$
\end{definition}
Furthermore, the \textit{period} $\pi$ of the two-point Weierstrass semigroup $H(Q_1, Q_2)$ is defined to be
$$ \pi := \min \{ k \in \mathbb{N} \setminus \{ 0 \} \mid k(Q_1 - Q_2) \; \text{is a principal divisor} \}. $$
With these definitions in place, we finally define the following function:
\begin{alignat*}{2}
  \tau_{Q_1, Q_2} : \; & \mathbb{Z} && \longrightarrow \mathbb{Z} \\
  & i && \longmapsto \min \{ j \mid (i, j) \in H(Q_1, Q_2) \}.
\end{alignat*}
More details on the map $\tau_{Q_1, Q_2}$ can be found in \cite{BT06}*{Proposition 14, Proposition 17}. However, for the convenience of the reader, in the following proposition we summarize some of its most important properties.
\begin{proposition}[\cite{BT06}*{Proposition 14, Proposition 17}]
  \label{prop:properties_tau}
  Let $\pi$ be the period of the two-point Weierstrass semigroup $H(Q_1, Q_2)$ and $g := g(\mathcal{X})$ be the genus of $\mathcal{X}$. Then:
  \begin{enumerate}
  \item $\tau_{Q_1, Q_2}$ is bijective, with inverse map $\tau^{-1}_{Q_1, Q_2} = \tau_{Q_2, Q_1}$;
  \item $-i \leq \tau_{Q_1, Q_2}(i) \leq 2g - i$ for all $i \in \mathbb{Z}$;
  \item $\tau_{Q_1, Q_2}(i + \pi) = \tau_{Q_1, Q_2}(i) - \pi$;
  \item $\sum_{i=c}^{\pi+c-1} (i + \tau_{Q_1, Q_2}(i)) = \pi g$ for all $c \in \mathbb{Z}$.
  \end{enumerate}
\end{proposition}
Moreover, as already observed in \cite{LV22}*{Proposition 2.7}, once the map $\tau_{Q_1, Q_2}$ is known, for all $j \in \mathbb{Z}$ the inverse map $\tau^{-1}_{Q_1, Q_2}(j)$ can also be directly and efficiently computed.

As the two following results show, the knowledge of the function $\tau_{Q_1, Q_2}$ is particularly useful for the computation of the order bound for a code $C_L(D, aQ_1 + bQ_2)^\perp$, since it allows the determination of the dimension of $L(aQ_1 + bQ_2)$, $a, b\in \mathbb{N}$, and provides an explicit expression for the set of $G$-non-gaps at $Q_1$ and the set of $G$-non-gaps at $Q_2$.

\begin{theorem}[\cite{TwoPointCodes}*{Theorem 9}]
  \label{theorem:dim_LG}
  Let $G = aQ_1 + bQ_2$ with $a, b \in \mathbb{N}$. The Riemann-Roch space $L(G)$ has dimension $|\{i \leq a \mid \tau_{Q_1, Q_2}(i) \leq b \}|$.
\end{theorem}

\begin{corollary}[\cite{TwoPointCodes}*{Corollary 10}]
  \label{cor:tau-non-gaps}
  Let $G = aQ_1 + bQ_2$ with $a, b \in \mathbb{N}$. Then $H(Q_1; G) = \{ i \in \mathbb{Z} \mid \tau_{Q_1, Q_2}(i) \leq b \}$ and $H(Q_2; G) = \{ i \in \mathbb{Z} \mid \tau_{Q_1, Q_2}^{-1}(i) \leq a \}$.
\end{corollary}

Finally, a consequence of Theorem \ref{theorem:dim_LG} that has been explicitly discussed in \cite{LV22}*{Corollary 2.10} is that one can fully determine the two-point Weierstrass semigroup $H(Q_1, Q_2)$ as
$$H(Q_1, Q_2) = \{ (i, j) \in \mathbb{Z}^2 \mid \tau_{Q_1, Q_2}(i) \leq j, \tau_{Q_1, Q_2}^{-1}(j) \leq i \}.$$
This means that the knowledge of the function $\tau_{Q_1, Q_2}$ is equivalent to the knowledge of the full two-point semigroup $H(Q_1, Q_2)$. 

\subsection{The Skabelund curve \texorpdfstring{$\Sq$}{Sq}}
\label{sec:skabelund}

In this section, we introduce the Skabelund maximal curve $\Sq$, constructed by D. Skabelund in \cite{S17} as a Kummer cover of the Suzuki curve. Let $s \in \N$, $s \geq 1$ and define $q_0 := 2^s$, $q := 2q_0^2$ and $m := q - 2q_0 + 1$. The Skabelund curve $\Sq$ is defined by affine equations
\begin{equation}
  \label{eq:skabelund_curve}
  \Sq : \begin{cases}
    y^q + y = x^{q_0}(x^q + x), \\
    t^m = x^q + x.
  \end{cases}
\end{equation}
We hereby recall some of the main properties of $\Sq$. More details can be found in \cite{GMQZ18} and \cite{S17}.
\begin{proposition}
  \label{prop:properties_skabelund_curve}
  The Skabelund curve $\Sq$ is maximal over $\F_{q^4}$, has genus $g(\Sq) = \frac{1}{2} q (q-1)^2$ and $q^5 - q^4 + q^3 + 1$ $\F_{q^4}$-rational points. The full automorphism group of $\Sq$ acts on the set of $\F_{q^4}$-rational points of $\Sq$ with two short orbits; one is non-tame of size $q^2+1$ consisting of exactly all $\F_q$-rational points and the other is tame of size $q^5 - q^4 + q^3 - q^2$ consisting of all $\F_{q^4}$-rational points that are not $\F_q$-rational.
\end{proposition}
We denote with $\mc{O}$ the short orbit of $\Sq$ consisting of all the $\F_q$-rational points and with $P_\infty$ the unique point at infinity of $\Sq$, which belongs to the orbit $\mc{O}$.  Moreover, we denote with $P_{(a,b,c)}$ the point of $\Sq$ with affine $(x,y,t)$-coordinates $(a,b,c)$. In particular, for less cumbersome notations, in Section \ref{sec:two_point_semigroup} we simply denote as $P$ the point $P_{(0,0,0)}\in \mc{O} \setminus \{P_\infty\}$.

In the following proposition (see \cite[Section 2]{BLM21}), we recall the principal divisors on $\Sq$ of the coordinate functions and of the following two functions:
$$ z := x^{2q_0+1} + y^{2q_0}, \quad w := x y^{2q_0} + z^{2q_0}. $$

\begin{proposition}[\cite{BLM21}*{Section 2}]
The principal divisors on $\Sq$ of the functions $x,y,t,z,w$ are:
  \label{prop:principaldivisors}
  \begin{equation}
  \label{eq:divisors}
  \begin{aligned}
    (x) & = mP + E_x - (q^2 - 2q q_0 + q) P_\infty, \\
    (y) & = m(q_0+1)P + E_y - (q^2 - q q_0 + q_0) P_\infty, \\
    (t) & = \sum_{a^q+a \ = \ 0 \ \wedge \ b^q+b \ = \ 0}P_{(a,b,0)}- q^2P_\infty,\\
    (z) & = m(2q_0+1) P + E_z - (q^2 - q + 2q_0) P_\infty, \\
    (w) & = (q^2 + 1)(P - P_\infty),
  \end{aligned}
\end{equation}
where $E_x, E_y, E_z$ are effective divisors whose support does not contain $P$ and $P_\infty$.
\end{proposition}

\section{The two-point Weierstrass semigroup \texorpdfstring{$H(P, P_\infty)$}{H(P, P_inf)}} 
\label{sec:two_point_semigroup}

In this paper, we are interested in using the order bound to compute a lower bound for the minimum distance of the AG codes $C_L(D, G)^\perp$ on $\Sq$ with
\begin{equation}
\label{eq:codedef}
    D:=\sum_{R\in \Sq(\F_{q^4})\setminus \{P,P_\infty\}}R \quad \mbox{and} \quad G:= aP + bP_\infty,
\end{equation}
for $a,b\in \mathbb{Z}_{>0}$.

To this aim, as pointed out in the final part of Section \ref{subsec:orderbound}, the purpose of this section is to explicitly determine the function $\tau_{P, P_\infty}$.

The points $P_\infty$ and $P$ lie in the same orbit under the action of the full automorphism group of $\Sq$, therefore $H(P)=H(P_\infty)$. The Weierstrass semigroup at every point of the Skabelund curve $\Sq$ is known (see \cite{BLM21}), hence we just recall the structure of $H(P_\infty)$ in the following proposition.  

\begin{proposition}[\cite{BLM21}]
  \label{prop:semigroups}
  The Weierstrass semigroup of $\Sq$ at $P_\infty$ is
  $$ H(P_\infty) = \langle q^2-2q q_0+q, q^2-q q_0 +q_0, q^2-q+2q_0, q^2, q^2+1 \rangle. $$
\end{proposition}

To the aim of determining the function $\tau_{P,P_\infty}$, we start now by computing the ring $\mc{R}(P, P_\infty)$ of regular functions outside $P$ and $P_\infty$, see \cref{def:ringregular}.

\begin{proposition}
  \label{prop:regular_functions}
  The ring of all $\F_{q^4}$-rational functions that are regular outside $P$ and $P_\infty$ is
  $$ \mc{R}(P, P_\infty) = \F_{q^4} [x, y, z, t, w, w^{-1}]. $$
\end{proposition}

\begin{proof}
  It is clear from \eqref{eq:divisors} that any function in $\F_{q^4} [x, y, z, t, w, w^{-1}]$ is regular outside $P$ and $P_\infty$. Conversely, from \eqref{eq:divisors} and \cref{prop:semigroups} it follows that the ring $\mc{R}(P_\infty)$ consisting of all $\F_{q^4}$-rational functions that are regular outside $P_\infty$ is
  $$ \mc{R}(P_\infty) = \F_{q^4} [x, y, z, t, w]. $$
  Hence, for any $f \in \mc{R}(P, P_\infty)$ there exists a suitable integer $k \geq 0$ such that $f w^k$ belongs to $\mc{R}(P_\infty)$. This shows that $f$ belongs to $\F_{q^4} [x, y, z, t, w, w^{-1}]$.
\end{proof}

\begin{lemma}
  \label{lem:period_tau}
  The period of the Weierstrass semigroup $H(P, P_\infty)$ is $q^2 + 1$.
\end{lemma}

\begin{proof}
  Assume by contradiction that there exists a function $f$ defined on $\Sq$ having principal divisor $(f) = k (P - P_\infty)$ for some $k \in \{ 1, \dots, q^2 \}$. Then $k$ must be a non-gap of the Weierstrass semigroup $H(P_\infty)$. The smallest non-zero element of $H(P_\infty)$ is $q^2-2qq_0+q$ (see \cref{prop:semigroups}), hence $q^2-2qq_0+q \leq k \leq q^2$ and we can write $k = q^2-2q q_0+q+j$ for some $j \in \{ 0, \dots, 2qq_0-q \}$. The function $w f^{-1}$ has principal divisor
  $$(w f^{-1}) = (q^2+1-k)(P - P_\infty) = (2qq_0-q+1-j)(P - P_\infty), $$
  hence $2qq_0-q+1-j$ must be a non-gap of the Weierstrass semigroup $H(P_\infty)$. This is not possible, as $0 < 2qq_0-q+1-j < 2qq_0-q+1 < q^2-2qq_0+q$.
\end{proof}

In the following, we denote the period of $H(P, P_\infty)$ with $\rho := q^2 + 1$.

\begin{lemma}
\label{lem:interval_i}
Let $i \in \mathbb{Z}$, $k := \left\lfloor\frac{i-1}{\rho}\right\rfloor$ and $r := i - k\rho - 1$. Then, $i$ can be written uniquely as
$$ i = (k+1)\rho - (a_t + ma_x + (q_0+1)ma_y + (2q_0+1)ma_z), $$
with $a_t, a_x, a_y, a_z \in \mathbb{Z}$ such that $0\leq a_t \leq m-1$ and
$$\begin{cases}
  0\leq a_x \leq q_0, a_y = 0, a_z = q_0 & \text{if } r < m(q_0+1), \\
  0\leq a_x \leq q_0-a_y, 0\leq a_y \leq 1, 0\leq a_z \leq q_0-1 & \text{otherwise}.
\end{cases}$$
\end{lemma}

\begin{proof}
  Observe that $k\rho+1 \leq i \leq (k+1)\rho$ and $0 \leq r \leq \rho-1$. Assume first $r < m(q_0+1)$ and notice that the condition
  $$ 0 \leq \rho-1 - (a_t + ma_x + (q_0+1)ma_y + (2q_0+1)ma_z) < m(q_0+1) $$
  holds for any choices of $a_t, a_x, a_y, a_z$ as in the assumptions. Moreover, suppose that $r$ can be expressed as:
  $$ r = \rho-1 - (a_t + ma_x + (2q_0+1)mq_0) = \rho-1 - (a_t^\prime + ma_x^\prime + (2q_0+1)mq_0) $$
  with $0\leq a_t, a_t^\prime \leq m-1$, $0\leq a_x, a_x^\prime \leq q_0$. Considering the above equation modulo $m$, it directly follows that $a_t = a_t^\prime$ and $a_x = a_x^\prime$. Hence, there are precisely $m(q_0+1)$ distinct integers in the interval $[0, m(q_0+1)-1]$, one for each possible choice of $a_t, a_x, a_y, a_z$, that can be written as $\rho-1 - (a_t + ma_x + (2q_0+1)mq_0)$. The first case of the statement follows. \\

  A similar argument can be repeated for the case $m(q_0+1) \leq r \leq \rho-1$; the condition
  $$ m(q_0+1) \leq \rho-1 - (a_t + ma_x + (q_0+1)ma_y + (2q_0+1)ma_z) \leq \rho-1 $$
  holds for any choices of $a_t, a_x, a_y, a_z$ as in the assumptions. Further, assume that $r$ can be expressed as:
  \begin{align*}
    r & = \rho-1 - (a_t + ma_x + (q_0+1)ma_y + (2q_0+1)ma_z) \\
      & = \rho-1 - (a_t^\prime + ma_x^\prime + (q_0+1)ma_y^\prime + (2q_0+1)ma_z^\prime)
  \end{align*}
  with $0\leq a_t, a_t^\prime \leq m-1$, $0\leq a_x \leq q_0-a_y, 0 \leq a_x^\prime \leq q_0-a_y^\prime, 0\leq a_y, a_y^\prime \leq 1, 0\leq a_z, a_z^\prime \leq q_0-1$. Considering the above equation first modulo $m$, then modulo $2q_0+1$, and finally modulo $q_0+1$, we conclude that $a_t = a_t^\prime, a_x = a_x^\prime, a_y = a_y^\prime, a_z = a_z^\prime$. There are $m(q+q_0)$ possible choices of $a_t, a_x, a_y, a_z$, as well as distinct integers in the interval $[m(q_0+1), \rho-1]$; we conclude that any integer in the interval $[m(q_0+1), \rho-1]$ can be expressed uniquely as $\rho-1 - (a_t + ma_x + (q_0+1)ma_y + (2q_0+1)ma_z)$ and the second case of the statement follows.
\end{proof}

\begin{theorem}
  \label{theorem:tau_skabelund}
  Let $i \in \mathbb{Z}$, $k:=\left\lfloor\frac{i-1}{\rho}\right\rfloor$ and write
\begin{equation*}
        i = (k+1)\rho - (a_t + ma_x + (q_0+1)ma_y + (2q_0+1)ma_z),
    \end{equation*}
    for a unique quadruple $(a_t, a_x, a_y, a_z) \in \mathbb{Z}^4$ such that $0\leq a_t \leq m-1$, $0\leq a_y \leq 1$, $0\leq a_x \leq q_0-a_y$ and $0\leq a_z \leq q_0$.
 Then
     $$ \tau_{P, P_\infty} (i) = (a_t q^2 + a_z (q^2-q+2q_0) + a_y (q^2-qq_0+q_0) + a_x (q^2-2qq_0+q)) - (k+1)\rho. $$
\end{theorem}

\begin{proof}
To prove the theorem, we define the map $\tilde{\tau} : \mathbb{Z} \to \mathbb{Z}$ such that $$\tilde{\tau}(i) = (a_t q^2 + a_z (q^2-q+2q_0) + a_y (q^2-qq_0+q_0) + a_x (q^2-2qq_0+q)) - (k+1)\rho$$
for all $i \in \mathbb{Z}$ and $k, a_t, a_x, a_y, a_z$ as in the assumptions, and we show that $\tilde{\tau}(i) = \tau_{P, P_\infty}(i)$ for all $i \in \mathbb{Z}$.

We start by showing that $\tilde{\tau}(i) \geq \tau_{P, P_\infty}(i)$ for all $i\in \mathbb{Z}$. To this aim, let $i \in \mathbb{Z}$, $k:=\left\lfloor\frac{i-1}{\rho}\right\rfloor$ as above and write $i$ as 
$$i = (k+1)\rho - (a_t + ma_x + (q_0+1)ma_y + (2q_0+1)ma_z),$$
for the suitable $(a_t, a_x, a_y, a_z) \in \mathbb{Z}^4$. 
Consider the function 
$$ f := w^{-(k+1)} t^{a_t} z^{a_z} y^{a_y} x^{a_x}. $$
Then, by Proposition \ref{prop:principaldivisors}, it follows that the principal divisor of $f$ is 
  \begin{alignat*}{2}
    (f) & = && -(k+1)(q^2 + 1)(P - P_\infty) + a_t\left( \sum_{a^q+a \ = \ 0 \ \wedge \ b^q+b \ = \ 0}P_{(a,b,0)} - q^2P_\infty\right) \\
    & && + a_z\left(m(2q_0+1) P + E_z - (q^2 - q + 2q_0) P_\infty\right) \\
    & && + a_y\left( m(q_0+1)P + E_y - (q^2 - q q_0 + q_0) P_\infty\right) \\
    & && + a_x\left(mP + E_x - (q^2 - 2q q_0 + q) P_\infty\right)\\
    & = && -i P + E - \tilde{\tau}(i) P_\infty,
  \end{alignat*}
  where $E$ is an effective divisor whose support do not contain $P$ and $P_\infty$. The above computation shows that $(i, \tilde{\tau}(i))$ belongs to $H(P, P_\infty)$ and thus $\tilde{\tau}(i) \geq \tau_{P, P_\infty}(i)$ by definition of $\tau_{P, P_\infty}$.

  For showing that the equality $\tilde{\tau}(i) = \tau_{P, P_\infty}(i)$ holds for all $i\in \mathbb{Z}$, we can now use Proposition \ref{prop:properties_tau}(4). In fact, we have just proved that $\tilde{\tau}(i) \geq \tau_{P, P_\infty}(i)$ for all $i \in \mathbb{Z}$ and therefore
  \begin{equation}
    \label{eq:sum_equal_pig}
    \sum_{i=c}^{\pi+c-1} (i + \tilde{\tau}(i)) \geq \sum_{i=c}^{\pi+c-1} (i + \tau_{P, P_\infty}(i)) = \pi g(\Sq)
  \end{equation}
  for all $c \in \mathbb{Z}$. To conclude, it is enough to check that the left side of equation \eqref{eq:sum_equal_pig} is equal to $\pi g(\Sq)$. We can choose $c = 0$ without loss of generality, hence we obtain
  \begin{equation*}
    \begin{split}
        &\sum_{i=0}^{\pi -1} (i + \tilde{\tau}(i)) = \\
        &= - \sum_{a_t=0}^{m-1} \sum_{a_z=0}^{q_0-1} \sum_{a_y=0}^{1} \sum_{a_x=0}^{q_0-a_y} ((a_t + ma_x + (q_0+1)ma_y + (2q_0+1)ma_z) +\\
         & - (a_t q^2 + a_z (q^2-q+2q_0) + a_y (q^2-qq_0+q_0) + a_x (q^2-2qq_0+q))) + \\
         & - \sum_{a_t=0}^{m-1} \sum_{a_x=0}^{q_0} ((a_t + ma_x + (2q_0+1)mq_0) - (a_t q^2 + q_0 (q^2-q+2q_0) + a_x (q^2-2qq_0+q)))\\
         & = - (1-q^2)\left((q_0+1)q_0\frac{(m-1)m}{2} + q_0^2\frac{(m-1)m}{2} + (q_0+1)\frac{(m-1)m}{2}\right) + \\
         &  - (m-(q^2-2qq_0+q))\left(\frac{q_0(q_0+1)}{2}q_0 m + \frac{(q_0-1)q_0}{2}q_0 m + \frac{q_0(q_0+1)}{2}m\right) + \\
         &  - ((2q_0+1)m-(q^2-q+2q_0))\left(\frac{(q_0-1)q_0}{2}(q_0+1)m + \frac{(q_0-1)q_0}{2}q_0 m + q_0(q_0+1)m\right) + \\
         &  - ((q_0+1)m-(q^2-qq_0+q_0))(q_0^2 m) = 16q_0^{10} - 16q_0^8 + 8q_0^6 - 4q_0^4 + q_0^2 = \pi g(\Sq).
    \end{split}
  \end{equation*}
\end{proof}

\begin{figure}[ht!]
    \centering

\caption{The two-point Weierstrass semigroup $H(P, P_\infty)$ of $\Sq$ for $s=1$, of period $\pi=65$. Only the pairs $(i, j) \in H(P, P_\infty)$ with $-2 \pi < i, j < 2 \pi$ are represented.}
\end{figure}

\section{Results and comparisons} \label{sec:results}

In this section, we present the results obtained for the AG codes $C_L(D, G)^\perp$ on $\Sq$ with
\begin{equation*}
    D:=\sum_{R\in \Sq(\F_{q^4})\setminus \{P,P_\infty\}}R \quad \mbox{and} \quad G:= aP + bP_\infty,
\end{equation*}
for $a,b\in \mathbb{Z}_{> 0}$ (see equation \eqref{eq:codedef}), where the parameter $s$ is set to be $s=1$. 
With this choice, we have $q_0 = 2$ and $q = 8$, so that the Skabelund curve $\Seight$ is maximal over $\F_{8^4}$ and has exactly $29185$ $\F_{8^4}$-rational points. Hence, the associated two-point AG codes have length $N := |\Seight(\F_{8^4})| - 2 = 29183$.
The order bound for the minimum distance of these codes is computed with the same algorithm as \cite{LV22}*{Algorithm 4.1}, which is inspired by \cite{TwoPointCodes}*{Algorithm 1}. Therefore, we refer to \cite{LV22}*{Section 4} for a thorough description of the algorithm and the technical details. Indeed, the assumptions that we made in Remark \ref{rem:restrictions} for computational purposes are essentially the same as those made in \cite{LV22}. In particular, this means that we only consider the codes with $a+b \leq 4g(\Seight) - 1$.

The results obtained show that in several cases the order bound significantly improves the Goppa bound. Moreover, we observed that, for all the two-point codes considered, the minimum distance is always at least that of the one-point code of the same dimension. 
In Table \ref{table:s1}, we denote with $d$ the order bound for the minimum distance of the two-point code $C_L(D, aP + bP_\infty)^\perp$ and with $d_1$ the order bound for the minimum distance of the best one-point code $C_L(D, b^\prime P_\infty)^\perp$ with the same dimension $k$. 
The table contains all the results, for $s=1$ and code length $N = 29183$, for which the difference between the estimates $d$ and $d_1$ is larger than or equal to 10. 
The four rows in bold of Table \ref{table:s1} mark the codes for which $d - d_1 = 20$, which is the largest value obtained for such difference.  

\begin{remark}
    Choosing the divisor $D$ to be $D:=\sum_{R\in \Sq(\F_{q})\setminus \{P,P_\infty\}}R$, i.e., such that its support contains precisely the points in $\Sq(\F_q)\setminus \{P,P_\infty\}$, one could in principle compute the order bound for two-point codes on $\Sq$ defined over $\F_q$, instead of $\F_{q^4}$, and hence compare such codes with two-point codes arising from the Suzuki curve and studied in \cite{TwoPointsSuzuki}. However, since the genus of $\Sq$ is much larger than $|\Sq(\F_q)|$, the order bound does not give a good estimate for the minimum distance in this case.
\end{remark}

\begin{center}
\begin{longtable}{c c c c c | c c c c c}
\caption{For $s=1$, the table contains the best possible estimates for the minimum distance $d$ and $d_1$, obtained with the order bound, for a two-point code $C_L(D, aP + bP_\infty)^\perp$ and a one-point code $C_L(D, b^\prime P_\infty)^\perp$ of a certain dimension $k$ and length $N = 29183$, respectively. 
} \label{table:s1} \\

\hline \multicolumn{1}{c}{$k$} & \multicolumn{1}{c}{$(a, b)$} & \multicolumn{1}{c}{$d$} & \multicolumn{1}{c}{$d_1$} & \multicolumn{1}{c}{$b^\prime$} & \multicolumn{1}{c}{$k$} & \multicolumn{1}{c}{$(a, b)$} & \multicolumn{1}{c}{$d$} & \multicolumn{1}{c}{$d_1$} & \multicolumn{1}{c}{$b^\prime$}\\ \hline 
\endfirsthead

\multicolumn{10}{c}%
{{\bfseries \tablename\ \thetable{} -- continued from previous page}} \\
\hline \multicolumn{1}{c}{$k$} & \multicolumn{1}{c}{$(a, b)$} & \multicolumn{1}{c}{$d$} & \multicolumn{1}{c}{$d_1$} & \multicolumn{1}{c}{$b^\prime$} & \multicolumn{1}{c}{$k$} & \multicolumn{1}{c}{$(a, b)$} & \multicolumn{1}{c}{$d$}  & \multicolumn{1}{c}{$d_1$} & \multicolumn{1}{c}{$b^\prime$}\\ \hline  
\endhead

\hline \multicolumn{10}{r}{{Continued on next page}} \\ \hline
\endfoot

\hline \hline
\endlastfoot

28860 & (1, 517) & 138 & 128 & 518 &  28933 & (1, 444) & 70 & 60 & 445 \\
28861 & (1, 516) & 138 & 128 & 517 & 28934 & (1, 443) & 70 & 60 & 444\\
28864 & (1, 513) & 134 & 124 & 514 & 28935 & (1, 442) & 70 & 60 & 443\\
28865 & (1, 512) & 134 & 124 & 513 & 28936 & (1, 441) & 70 & 60 & 442\\
28866 & (1, 511) & 134 & 124 & 512 & 28938 & (1, 439) & 65  & 50 & 440\\ 
28868 & (1, 509) & 130 & 120 & 510 &  28939 & (1, 438) & 65 & 50 & 439\\
28869 & (1, 508) & 130 & 120 & 509 &  28940 & (1, 437) & 65 & 50 & 438\\
28870 & (1, 507) & 130 & 120 & 508 &  28941 & (1, 436) & 65 & 50 & 437\\
28871 & (1, 506) & 130 & 120 & 507 &  28942 & (1, 435) & 65 & 50 & 436\\
28874 & (1, 503) & 124 & 114 & 504 &  28943 & (1, 434) & 65 & 50 & 435\\
28875 & (1, 502) & 124 & 114 & 503 &  28944 & (1, 433) & 65 & 50 & 434\\
28876 & (1, 501) & 124 & 114 & 502 &  28945 & (1, 432) & 65 & 50 & 433\\
28878 & (1, 499) & 120 & 110 & 500 &  28946 & (1, 431) & 65 & 50 & 432\\
28879 & (1, 498) & 120 & 110 & 499 & \textbf{28948} & \textbf{(6, 424)}  &\textbf{60}  & \textbf{40} & \textbf{430}\\
28880 & (1, 497) & 120&   110 & 498 & \textbf{28949} & \textbf{(6, 423)}  & \textbf{60}   & \textbf{40} & \textbf{429}\\
28881 & (1, 496) & 120&   110 & 497 & \textbf{28950} & \textbf{(6, 422)}  & \textbf{60}   & \textbf{40} & \textbf{428}\\
28884 & (1, 493) & 114&   104 & 494 & \textbf{28951} & \textbf{(6, 421)}  & \textbf{60}  & \textbf{40} & \textbf{427}\\
28885 & (1, 492) & 114&  104 & 493 & 28952 & (1, 425) & 55  & 40 & 426 \\
28886 & (1, 491) & 114&  104 & 492 & 28953 & (1, 424) & 55  & 40 & 425 \\
28888 & (1, 489) & 110&  100 & 490 & 28954 & (1, 423) & 55  & 40 & 424 \\
28889 & (1, 488) & 110&  100 & 489 & 28955 & (1, 422) & 55  & 40 & 423 \\
28890 & (1, 487) & 110&  100 & 488 & 28956 & (1, 421) & 55  & 40 & 422 \\
28891 & (1, 486) & 110&  100 & 487 & 28957 & (1, 420) & 50  & 40 & 421 \\
28898 & (1, 479) & 100&  90&   480 & 28958 & (1, 419) & 50  & 40 & 420 \\
28899 & (1, 478) & 100&  90&   479 & 28959 & (1, 418) & 50  & 40 & 419 \\
28900 & (1, 477) & 100&  90&   478 & 28960 & (1, 417) & 50  & 40 & 418 \\
28901 & (1, 476) & 100&  90&   477 & 28961 & (1, 416) & 50  & 40 & 417 \\
28908 & (1, 469) & 90 & 80 & 470  & 28978 & (11, 389) &  40  & 30 & 400\\
28909 & (1, 468) & 90 & 80 & 469  & 28979 & (11, 388) &  40  & 30 & 399\\
28910 & (1, 467) & 90 & 80 & 468  & 28980 & (11, 387) &  40  & 30 & 398\\
28911 & (1, 466) & 90 & 80 & 467  & 28981 & (11, 386) &  40  & 30 & 397\\
28923 & (1, 454) & 79 & 65 & 455 & 28997 & (56, 324) &  30 & 20 & 380\\
28924 & (1, 453) & 79 & 64 & 454 & 28998 & (56, 323) &  30 & 20 & 379\\
28925 & (1, 452) & 79 & 64 & 453 & 28999 & (56, 322) &  30 & 20 & 378\\
28926 & (1, 451) & 79 & 64 & 452 & 29000 & (56, 321) &  30 & 20 & 377\\
28927 & (1, 450) & 75 & 64 & 451 & 29001 & (56, 320) &  30 & 20 & 376\\
28928 & (1, 449) & 75 & 60 & 450 & 29002 & (56, 319) &  30 & 20 & 375\\
28929 & (1, 448) & 75 & 60 & 449 & 29003 & (56, 318) &  30 & 20 & 374\\
28930 & (1, 447) & 75 & 60 & 448 & 29004 & (56, 317) &  30 & 20 & 373\\
28931 & (1, 446) & 75 & 60 & 447 & 29005 & (56, 316) &  30 & 20 & 372\\
28932 & (1, 445) & 70 & 60 & 446
\end{longtable} 
\end{center}

\section*{Acknowledgements}
The second author is funded by the project ``Metodi matematici per la firma digitale ed il cloud computing" (Programma Operativo Nazionale (PON) ``Ricerca e Innovazione" 2014-2020, University of Perugia) and his research was partially supported  by the Italian National Group for Algebraic and Geometric Structures and their Applications (GNSAGA - INdAM).

\end{document}